\documentclass{amsart}
\usepackage{color}

\newtheorem{theorem}{Theorem}[section]

\newtheorem{lemma}[theorem]{Lemma}
\newtheorem{proposition}[theorem]{Proposition}
\theoremstyle{definition}
\newtheorem{definition}[theorem]{Definition}

\theoremstyle{remark}
\newtheorem{remark}[theorem]{Remark}

\numberwithin{equation}{section}

\begin{document}

\title[ On two  Bloch type theorems   for quaternionic  slice regular functions]
{On two  Bloch type theorems   for quaternionic slice regular functions}

\author[Z.  Xu]{Zhenghua Xu}
\address{Zhenghua Xu, Department of Mathematics, University of Science and
Technology of China, Hefei 230026, China}
\email{xzhengh$\symbol{64}$mail.ustc.edu.cn }
\author[X.  Wang]{Xieping Wang}
\address{Xieping Wang, Department of Mathematics, University of Science and
Technology of China, Hefei 230026,China}
\email{pwx$\symbol{64}$mail.ustc.edu.cn}
\thanks{This work was supported by the NNSF  of China (11071230), RFDP (20123402110068).}

\keywords{Quaternion, Slice regular functions, Bloch constant}
\subjclass[2010]{30G35.}

\begin{abstract}
In this paper  we prove two Bloch  type theorems   for quaternionic slice regular functions.  We  first discuss the injective and covering properties of  some  classes of slice regular functions from  slice regular Bloch spaces and  slice regular Bergman spaces, respectively. And then we show  that there exits a universal ball contained in the image of the open unit ball $\mathbb{B}$ in quaternions $\mathbb{H}$ through the slice regular rotation $\widetilde{f}_{u}$ of each slice regular function $f:\overline{\mathbb{B}}\rightarrow \mathbb{H}$  with  $f'(0)=1$ for some $u\in \partial\mathbb{B}$.

\end{abstract}

\maketitle

\section{Introduction}
Let $\mathbb C$ be the complex plane and $D(z_0, R)$ the open  disc centred at $z_0\in\mathbb C$ with radius $R>0$. For simplicity, we denote by $\mathbb D$ the open unit disc $D(0, 1)$.
Let $H(\mathbb{D})$ denote the class of holomorphic functions on $\mathbb D$. Given a function $F\in H(\mathbb{D})$, we  define $B_{F}$ to be the least upper bound of all positive numbers $R>0$ such that there exists a number $z_{0}\in \mathbb{ C}$ and a domain $\Omega \subset  \mathbb{D}$
which is mapped conformally by $F$ onto $D(z_0, r)$. 
The Bloch constant $\mathbf{B}_h$ is defined to be
$$\mathbf{B}_h := \inf\big\{B_{F} : F\in H(\mathbb{D}),   F'(0) = 1\big\}. $$
The Bloch's theorem asserts that  $\mathbf{B}_h>0$ in \cite{Bloch1}. In \cite{Landau2}, Landau showed that
\begin{equation}\label{Landau-Bloch}
\mathbf{B}_h=\inf\big\{B_{F}: F\in \mathcal{B}_h, F(0)=0,  F'(0) = 1\big\},
\end{equation}
where $\mathcal{B}_h$ denotes the class of holomorphic functions $F\in H(\mathbb{D})$  with the Bloch seminorm
$$\|F\|_{\mathcal{B}_h}:=\sup_{z\in \mathbb{D}}(1-|z|^{2})|F'(z)|\leq1.$$

  In \cite{Ahlfors},  Ahlfors proved  that $\mathbf{B}_h\geq3/4$ using his far-reaching generalization of the classical Schwarz-Pick lemma. The strict inequality $\mathbf{B}_h>3/4$ was established by  Heins \cite{Heins1} and Pommerenke \cite{Pommerenke}, respectively. In \cite{Bonk}, Bonk gave a new and remarkable proof of this result and  improved slightly it using   the following theorem, known as Bonk's distortion theorem.
\begin{theorem}\label{Bonk}
Let $F\in\mathcal{B}_h$ be such that $F'(0) = 1.$
Then the inequality
$${\rm{Re}}\, F'(z)\geq \frac{1-\sqrt{3}|z|}{(1-|z|/\sqrt{3})^{3}}$$
holds for all $z\in D(0, 1/\sqrt{3})$.
\end{theorem}

The finding of the precise value of $\mathbf{B}_h$ is known to be the number one open problem in the geometric function theory of one complex variable since the confirmation of the Bieberbach conjecture by de Branges in 1985 \cite{Branges}. To the authors' knowledge, the best lower estimate for $\mathbf{B}_h$ is by now in \cite{Chen}.  The Bloch's theorem  does not hold  for general holomorphic mappings of several complex variables. The counterexample can be found in \cite{Harris}. Thus, one needs to restrict the class of mappings to a more specific subclass to obtain a Bloch's theorem. One of the well-known subclasses is the class of K-quasiregular mappings \cite{Wu}.  Recently, the Bloch's theorem in the bicomplex number setting has been investigated  successfully in \cite{Rochon}.

In this paper,  we first establish  the quaternionic analogues  of Bonk's distortion theorem and   Bloch  type theorem for slice regular  functions. The theory of slice regular  functions over quaternions was initiated recently by Gentili and Struppa \cite{GS1,GS2}. It is significantly different from the more classical theory of regular functions in the sense of Cauchy-Fueter and has  elegant  applications to the functional calculus for noncommutative operators \cite{Co2}, Schur analysis \cite{ACS} and the construction and classification of orthogonal complex structures on dense open subsets of $\mathbb R^4$ \cite{GSS2014}. For the detailed up-to-date theory, we refer the reader to the monographs \cite{GSS, Co2}.

In order to formulate precisely our main results, we first introduce a few  notations. Let $\mathbb H$ be the skew field of quaternions and $\mathbb{S}$ the unit sphere of purely imaginary quaternions, i.e.
$$\mathbb{S}=\{q\in \mathbb{H}: q^{2}=-1\}.$$
For each $R>0$ and each point $q_0\in\mathbb H$, we set $$B(q_0, R):=\big\{q \in \mathbb H:|q-q_0|<R\big\},$$
the Euclidean ball centred at $q_0$ of radius $R$, and for each $I\in\mathbb S$, we denote by $B(0, R)_{I}$ the intersection $B(0, R)\cap \mathbb{C}_{I}$. For simplicity, we denote by $\mathbb B$ the open unit disc $B(0, 1)$. Also, we denote  by $\mathcal{B}_r$ the slice regular Bloch space on $\mathbb B$, that is, the space of slice regular functions $f$ on $\mathbb B$ with the Bloch seminorm
$$\|f\|_{\mathcal{B}_r}:=\sup_{q\in \mathbb B}(1-|q|^{2})|f'(q)|<\infty,$$
where $f'$ is the slice derivative of $f$.  For each slice regular function $f$ on $\mathbb B$, we  define $B_{f}$ to be the least upper bound of all positive numbers $R>0$ such that there exists a number $q_{0}\in \mathbb H$ and a domain $\Omega \subset  \mathbb B$ such that
the restriction
$$\left.f\right|_{\Omega}:\Omega\rightarrow B(q_0, R)$$
is a homeomorphism.
Similar to $(\ref{Landau-Bloch})$, we define the Bloch constant $\mathbf{B}_r$ for $\mathcal{B}_r$ as
$$\mathbf{B}_r=\inf\big\{B_{f}: f\in \mathcal{B}_r, f(0)=0,  f'(0) = 1, \|f\|_{\mathcal{B}_r}\leq 1\big\}.$$
Now we can   show that $\mathbf{B}_r>0.23$ by proving the following result.
\begin{theorem}\label{Bloch-Landau-I-type-Bloch}
 Let $f\in \mathcal{B}_r$  be such that $f(0)=0, f'(0)=1$ and $||f||_{\mathcal{B}_r}\leq 1$. Then
 \begin{enumerate}
   \item [(a)] the inequality
   \begin{equation}\label{Bonk-ineq}
   {\rm{Re}}\, f'(q)\geq \frac{1-\sqrt{3}|q|}{(1-|q|/\sqrt{3})^{3}}
   \end{equation}
holds for all $q\in B(0, 1/\sqrt{3})$. In particular, $\left.f\right|_{B(0, 1/\sqrt{3})_I}: B(0, 1/\sqrt{3})_I\rightarrow \mathbb H$ is injective for every $I\in\mathbb S$.
   \item [(b)]
   $f$ is injective on $B(0, r_{\mathbb B})$ and $B(0, R_{\mathbb B})\subset f\big(B(0, r_{\mathbb B})\big)$, where $$ r_{\mathbb B}=\sup_{0<r<1}\bigg\{ \frac{1}{2}\log \frac{1+r}{1-r}-\sqrt{\frac{1}{4}\Big(\log\frac{1+r}{1-r}\Big)^{2}-r^{2}} \bigg\}
\approx 0.3552,$$
and
$$R_{\mathbb B}=\sup_{0<r<1}\bigg\{\frac{1}{2r^{2}}\log \frac{1+r}{1-r}\bigg(\frac{1}{2}\log \frac{1+r}{1-r}-\sqrt{\frac{1}{4}\Big(\log\frac{1+r}{1-r}\Big)^{2}-r^{2}}\bigg)^{2}  \bigg\}\approx 0.2308.$$
 Moreover,
 $$ B(0,\sqrt{3}/4) \subset f\big(B(0,1/\sqrt{3})\big).$$
 \end{enumerate}
\end{theorem}

  We can also prove a Bloch  type theorem  for   functions from  slice regular Bergman space $A^{p}(\mathbb B) (1\leq p<+\infty)$ of the second  kind introduced in \cite{CGLSSO}, that is, the Hilbert space of    slice regular functions $f$ on $\mathbb B$ with the norm
$$ \|f\|_{A^{p}}:=\sup_{I\in \mathbb{S}} \bigg(\int_{\mathbb{B}_{I}}|f_{I}(z)|^{p}d\sigma_{I}(z)\bigg)^{\frac{1}{p}}<\infty,$$
where $d\sigma_{I}$ stands  for the normalised Lebesgue measure on the plane $\mathbb{C}_{I}$, i.e  $$d\sigma_{I}(z)=\frac{1}{\pi}dxdy$$
for all $z=x+yI\in\mathbb{C}_{I}$.
\begin{theorem}\label{Bloch-Landau-I-type-Bergman}
 Let $p\in [1,+\infty)$ and $f\in A^{p}(\mathbb B)$ be such that $f(0)=0, f'(0)=1$ and $\|f\|_{A^{p}}\leq1$. Then $f$ is injective on   $B(0, r_p)$ and $B(0, R_p)\subset f\big(B(0, r_p)\big)$, where
$$ r_p=\sup_{0<r<1}\bigg\{ (1-r)^{-\frac{2}{p}}-\sqrt{(1-r)^{-\frac{4}{p}}-r^{2}}\bigg\},$$
and
$$R_p=\sup_{0<r<1}\bigg\{r^{-2}(1-r)^{-\frac{2}{p}}
\bigg((1-r)^{-\frac{2}{p}}-\sqrt{(1-r)^{-\frac{2}{p}}-r^{2}}\bigg)^{2}  \bigg\}.$$
\end{theorem}

 To prove Theorems \ref{Bloch-Landau-I-type-Bloch} and \ref{Bloch-Landau-I-type-Bergman},   we   need to  establish  a  quaternionic analogue  of a classical result due to Landau (see \cite{Landau2,Landau} or \cite[pp. 36--39]{Heins}) for  slice regular functions.    Let $\alpha\in (0,1)$. Define $\mathcal{F}_{\alpha}$ to be a family of slice regular functions  given by
$$\mathcal{F}_{\alpha}:=\big\{f:\mathbb B\rightarrow\mathbb B\ \big|\   f \ \mbox{is  slice regular such that}\ f(0)=0,\ |f'(0)|=\alpha\big\},$$
and for every $f\in \mathcal{F}_{\alpha}$, denote by $r_{\alpha}(f)$ the  radius  of the largest ball $B(r)$ on which $f$ is injective. Namely,
$$r_{\alpha}(f)=\sup\big\{r : \left.f\right|_{B(r)} \; \mbox{is injective} \big\}.$$
Set
$$r(\mathcal{F}_{\alpha})=\inf\big\{r_{\alpha}(f) : f\in \mathcal{F}_{\alpha}\big\}.$$
From \cite[Theorem 8.13]{GSS} it follows that for each $f\in\mathcal{F}_{\alpha}$, the real differential of $f$ at the origin $0$ is non-degenerate, and thus the inverse function theorem implies that $r_{\alpha}(f)$ is  positive. Indeed, our   result  shows that $r(\mathcal{F}_{\alpha})$ is a positive number as well.

\begin{theorem}\label{Main result-Landau}
With notations as above, the following statements hold:
\begin{enumerate}
\item [(a)] for each $\alpha\in(0,1)$,
$$r(\mathcal{F}_{\alpha})=\frac{\alpha}{1+\sqrt{1-\alpha^2}};$$

\item [(b)] for each $\alpha\in(0,1)$ and each $f\in \mathcal{F}_{\alpha}$,
$$r_{\alpha}(f)\geq r(\mathcal{F}_{\alpha})$$
with equality if and only if
\begin{equation}\label{Extremal-Cover}
f(q)=q\big(1-q\alpha e^{-I\theta}\big)^{-\ast}\ast\big(q-\alpha e^{I\theta}\big)v
\end{equation}
for some $I\in \mathbb S$, $\theta\in \mathbb R$ and $v\in\partial \mathbb B$;

\item [(c)] for each $0<r\leq r(\mathcal{F}_{\alpha})$,
$$\bigcap_{f\in \mathcal{F}_{\alpha}}f\big(B(0, r)\big)=B\big(0, R_{\alpha}(r)\big),$$
where
$$R_{\alpha}(r)=r\frac{\alpha-r}{1-\alpha r};$$

\item [(d)] for each $\alpha\in(0,1)$, $r\in\big(0, r(\mathcal{F}_{\alpha})\big)$ and $f\in \mathcal{F}_{\alpha}$,
    $$B\big(0, R_{\alpha}(r)\big)\subset f\big(B(0, r)\big).$$
    Moreover, $B\big(0, R_{\alpha}(r)\big)$ is the largest ball contained in the image set $f\big(B(0, r)\big)$ if and only if $f$ is of the form $(\ref{Extremal-Cover})$.
\end{enumerate}
\end{theorem}

%

In \cite{Bloch1}, Bloch  discovered a theorem  stating  that
if $F$ is a holomorphic function on   the closed unit disc $\{z\in \mathbb{C}: |z|\leq1\}$ such that $|F'(0)| = 1$, then the image domain contains discs of radius  $3/2-\sqrt{2}$.  More recently,  the corresponding  result has been  established  for square-integrable monogenic functions  in the open   ball of the paravector space $\mathbb R^3$ with values in $\mathbb R^3\subset \mathbb{H}$ as well (see \cite[Theorem 8]{Morais}), and for a kind of  regular translations of  slice regular functions on $\mathbb B$  in  the quaternionic setting  (see \cite[Theorem 6]{Rocchetta}), respectively. In this paper, we also present another version of  Bloch type theorem for slice regular functions. To this purpose, we need to consider the \textit{slice regular rotation} $\widetilde{f}_{u}$, instead of the \textit{regular translation} in
\cite{Rocchetta}, of a slice regular function $f(q)=\sum_{n=0}^{\infty}q^na_n$ on $\overline{\mathbb{B}}$ given by
$$ \widetilde{f}_{u}(q)=\sum\limits_{n=0}^{\infty}q^nu^{n}a_n,$$
for some  constant $u\in \partial\mathbb{ B}$.

\begin{theorem}\label{Bloch-Landau-II-type}
 Let $f$ be slice regular on  $\overline{\mathbb{B}}$ such that $f'(0)=1$.
Then there exists $u\in \partial\mathbb{B} $ such that the image of $\mathbb B$
under the slice regular rotation $\widetilde{f}_{u}$ of $f$ contains an open ball of radius $5/2-\sqrt{6}$.
\end{theorem}


The remaining part of this paper is organized as follows. In Sect. $2$, we set up basic notations and give some preliminary results from the theory of slice regular functions over quaternions.  In Sect. $3$, we shall give some useful lemmas, among which  the Landau's lemma  is established by means of one Lindel\"{o}f type inequality  for slice regular functions.   Sect. $4$ is  devoted to the proofs of Theorems  \ref{Bloch-Landau-I-type-Bloch}, \ref{Bloch-Landau-I-type-Bergman}, and \ref{Main result-Landau}. Finally, Theorem   \ref{Bloch-Landau-II-type} will be proved in Sect. $5$.

\section{Preliminaries}
In this section we recall some necessary definitions and preliminary results on slice regular functions.
To have a more complete insight on the theory, we refer the reader to the monograph \cite{GSS}.

Let $\mathbb H$ denote the non-commutative, associative, real algebra of quaternions with standard basis $\{1,\,i,\,j, \,k\}$,  subject to the multiplication rules
$$i^2=j^2=k^2=ijk=-1.$$
 Every element $q=x_0+x_1i+x_2j+x_3k$ in $\mathbb H$ is composed by the \textit{real} part ${\rm{Re}}\, (q)=x_0$ and the \textit{imaginary} part ${\rm{Im}}\, (q)=x_1i+x_2j+x_3k$. The \textit{conjugate} of $q\in \mathbb H$ is then $\bar{q}={\rm{Re}}\, (q)-{\rm{Im}}\, (q)$ and its \textit{modulus} is defined by $|q|^2=q\overline{q}=|{\rm{Re}}\, (q)|^2+|{\rm{Im}}\, (q)|^2$. We can therefore calculate the multiplicative inverse of each $q\neq0$ as $ q^{-1} =|q|^{-2}\overline{q}$.
 Every $q \in \mathbb H $ can be expressed as $q = x + yI$, where $x, y \in \mathbb R$ and
$$I=\dfrac{{\rm{Im}}\, (q)}{|{\rm{Im}}\, (q)|}$$
 if ${\rm{Im}}\, q\neq 0$, otherwise we take $I$ arbitrarily such that $I^2=-1$.
Then $I $ is an element of the unit 2-sphere of purely imaginary quaternions
$$\mathbb S=\big\{q \in \mathbb H:q^2 =-1\big\}.$$

For every $I \in \mathbb S $ we will denote by $\mathbb C_I$ the plane $ \mathbb R \oplus I\mathbb R $, isomorphic to $ \mathbb C$, and, if $\Omega \subset \mathbb H$, by $\Omega_I$ the intersection $ \Omega \cap \mathbb C_I $. Also, for $R>0$, we will denote the open ball centred at $q_{0}\in \mathbb{H}$ with radius $R$ by
$$B(q_{0},R)=\big\{q \in \mathbb H:|q-q_{0}|<R\big\}.$$
And let $\mathbb B$ denote the open unit ball $B(0,1)$ for simplicity.

We can now recall the definition of slice regularity.
\begin{definition} \label{de: regular} Let $\Omega$ be a domain in $\mathbb H$. A function $f :\Omega \rightarrow \mathbb H$ is called \emph{slice} \emph{regular} if, for all $ I \in \mathbb S$, its restriction $f_I$ to $\Omega_I$ is \emph{holomorphic}, i.e., it has continuous partial derivatives and satisfies
$$\bar{\partial}_I f(x+yI):=\frac{1}{2}\left(\frac{\partial}{\partial x}+I\frac{\partial}{\partial y}\right)f_I (x+yI)=0$$
for all $x+yI\in \Omega_I $. We will denote by $\mathcal{SR}(\mathbb{B})$ the set of slice regular functions on $\mathbb B$.
 \end{definition}

A wide class of examples of regular functions is given by polynomials and power series. Indeed, a function $f$ is slice regular on an open ball $B(0, R)$ if and only if $f$ admits a power series expansion $f(q)=\sum_{n=0}^{\infty}q^na_n$ converging absolutely and uniformly on every compact subset of $B(0, R)$.
As shown in \cite {CGSS}, the natural   domains of definition  of slice regular functions are the  so-called  symmetric slice domains.

\begin{definition} \label{de: domain}
Let $\Omega$ be a domain in $\mathbb H $.

1. $\Omega$ is called a \textit{slice domain}  if it intersects the real axis and if  for any $I \in \mathbb S $, $\Omega_I$  is a domain in $ \mathbb C_I $.

2. $\Omega$ is called an \textit{axially symmetric domain} if for any point  $x + yI \in \Omega$, with  $x,y \in \mathbb R $ and $I\in \mathbb S$, the entire two-dimensional sphere $x + y\mathbb S$ is contained in $\Omega $.
\end{definition}

A domain in $\mathbb H$ is called a \textit{symmetric slice domain} if it is not only a slice domain, but also an axially symmetric domain. From now on, we will omit the term `slice' when referring to slice regular functions and will focus mainly on regular functions on an open ball $B(0, R)$  which is a typical axially symmetric slice domain.
For regular functions the natural definition of derivative is given by the following (see \cite{GS1,GS2}).
\begin{definition} \label{de: derivative}
Let $f :\mathbb{B} \rightarrow \mathbb H$  be a regular function. For each $I\in\mathbb S$, the \textit{$I$-derivative} of $f$ at $q=x+yI$
is defined by
$$\partial_I f(x+yI):=\frac{1}{2}\left(\frac{\partial}{\partial x}-I\frac{\partial}{\partial y}\right)f_I (x+yI)$$
on $\mathbb{B} _I$. The \textit{slice derivative} of $f$ is the function $f'$ defined by $\partial_I f$ on $\mathbb{B}_I$ for all $I\in\mathbb S$.
 \end{definition}
The definition is well-defined because, by direct calculation, $\partial_I f=\partial_J f$ in $\mathbb{B} _I\cap \mathbb{B} _J$ for any choice of $I$, $J\in\mathbb S$. Furthermore, notice that the operators $\partial_I$ and $\bar{\partial}_I $ commute, and $\partial_I f=\frac{\partial f}{\partial x}$ for regular functions. Therefore, the slice derivative of a regular function is still regular so that we can iterate the differentiation to obtain the $n$-th
slice derivative
$$\partial^{n}_I f(x+yI)=\frac{\partial^{n} f}{\partial x^{n}}(x+yI),\quad\,\forall \,\, n\in \mathbb N. $$

In what follows, for the sake of simplicity, we will denote the $n$-th slice derivative by $f^{(n)}$ for every $n\in \mathbb N$.

In the theory of regular functions,  the following splitting lemma (see \cite{GS2}) relates closely slice regularity to classical holomorphy.
\begin{lemma}[Splitting Lemma]\label{eq:Splitting}
Let $f$ be a regular function on $\mathbb{B} $. Then for any $I\in \mathbb S$ and any $J\in \mathbb S$ with $J\perp I$, there exist two holomorphic functions $F,G:\mathbb{B} _I\rightarrow \mathbb C_I$ such that
$$f_I(z)=F(z)+G(z)J,\qquad \forall\ z=x+yI\in \mathbb{B} _I.$$
\end{lemma}

Since the regularity does not keep under point-wise product of two regular functions a new multiplication operation, called the regular product (or $\ast$-product), appears via  a suitable modification of the usual one subject to noncommutative setting. The regular product plays a key role in the theory of slice regular functions. On open balls centred at the origin, the $\ast$-product of two regular functions is defined by means of their power series
expansions (see, e.g., \cite{GSS1,CGSS}).

\begin{definition}\label{R-product}
Let $f$, $g:\mathbb{B} \rightarrow \mathbb H$ be two regular functions and let
$$f(q)=\sum\limits_{n=0}^{\infty}q^na_n,\qquad g(q)=\sum\limits_{n=0}^{\infty}q^nb_n$$
be their series expansions. The regular product (or $\ast$-product) of $f$ and $g$ is the function defined by
$$f\ast g(q)=\sum\limits_{n=0}^{\infty}q^n\bigg(\sum\limits_{k=0}^n a_kb_{n-k}\bigg)$$
and it is regular on $\mathbb{B} $.
\end{definition}
Notice that the $\ast$-product is associative and is not, in general, commutative. Its connection with the usual pointwise product is clarified by the following result \cite{GSS1,CGSS}.
\begin{proposition} \label{prop:RP}
Let $f$ and $g$ be  regular on $\mathbb{B} $. Then for all $q\in \mathbb{B} $,
$$f\ast g(q)=
\left\{
\begin{array}{lll}
f(q)g(f(q)^{-1}qf(q)) \qquad \,\,if \qquad f(q)\neq 0;
\\
\qquad  \qquad  0\qquad  \qquad \qquad if \qquad f(q)=0.
\end{array}
\right.
$$
\end{proposition}
We remark that if $q=x+yI$ and $f(q)\neq0$, then $f(q)^{-1}qf(q)$ has the same modulus and same real part as $q$. Therefore $f(q)^{-1}qf(q)$ lies in the same 2-sphere $x+y\mathbb S$ as $q$. Notice that a zero $x_0+y_0I$ of the function $g$ is not necessarily a zero of $f\ast g$, but some element on the same sphere $x_0+y_0\mathbb S$ does. In particular, a
real zero of $g$ is still a zero of $f\ast g$. To present a characterization of the structure of the zero set of a regular function $f$, we need to introduce the following functions.
\begin{definition} \label{de: R-conjugate}
Let $ f(q)=\sum\limits_{n=0}^{\infty}q^na_n $ be a regular function on $\mathbb{B} $. We define the \emph{regular conjugate} of $f$ as
$$f^c(q)=\sum\limits_{n=0}^{\infty}q^n\bar{a}_n,$$
and the \emph{symmetrization} of $f$ as
$$f^s(q)=f\ast f^c(q)=f^c\ast f(q)=\sum\limits_{n=0}^{\infty}q^n\bigg(\sum\limits_{k=0}^n a_k\bar{a}_{n-k}\bigg).$$
Both $f^c$ and $f^s$ are regular functions on $\mathbb{B} $.
\end{definition}
We are now able to define the inverse element of a regular function $f$ with respect to the $\ast$-product. Let $\mathcal{Z}_{f^s}$ denote the zero set of the symmetrization $f^s$ of $f$.
\begin{definition} \label{de: R-Inverse}
Let $f$ be a regular function on $\mathbb{B}$. If $f$ does not vanish identically, its \emph{regular reciprocal} is the function defined by
$$f^{-\ast}(q):=f^s(q)^{-1}f^c(q)$$
and it is regular on $\mathbb{B}  \setminus \mathcal{Z}_{f^s}$.

\end{definition}
The following result shows that the regular quotient is nicely related to the pointwise quotient (see \cite{Stop1, Stop2}).
\begin{proposition} \label{prop:Quotient Relation}
Let $f$ and $g$ be regular  on  $\mathbb{B} $. Then for all $q\in \mathbb{B}  \setminus \mathcal{Z}_{f^s}$, $$f^{-\ast}\ast g(q)=f(T_f(q))^{-1}g(T_f(q)),$$
where $T_f:\mathbb{B}  \setminus \mathcal{Z}_{f^s}\rightarrow \mathbb{B}  \setminus \mathcal{Z}_{f^s}$ is defined by
$T_f(q)=f^c(q)^{-1}qf^c(q)$. Furthermore, $T_f$ and $T_{f^c}$ are mutual inverses so that $T_f$ is a diffeomorphism.
\end{proposition}

We now recall a useful  result for regular functions \cite{SSthe,GSopen}.
\begin{theorem}\label{open-map}
Let $\Omega$ be a  symmetric slice domain and let $f: \Omega\rightarrow  \mathbb{H}$ be a nonconstant regular function. If $U$ is an axially symmetric open subset of $\Omega$, then $f(U)$ is open.
In particular, the image $f(\Omega)$ is open.
\end{theorem}


\section{Some lemmas}
In this section, we shall give some useful lemmas, which will be used in Sections 4 and 5. We begin with the following simple proposition.
\begin{proposition}\label{open}
Let $\Omega\subset \mathbb{H}$ be a bounded domain and $f:\Omega\rightarrow \mathbb{H}$  a continuous function such that  $ f(\Omega)$ is open in $\mathbb H$. Let  $a\in \Omega$ be a point such that
\begin{equation}\label{size-condition}
s:=\liminf_{q\rightarrow\partial \Omega}|f(q)-f(a)|>0.
\end{equation}
Then $ B(f(a), s)\subset f(\Omega)$.
\end{proposition}
\begin{proof}
For each point $w$ on the boundary $\partial f(\Omega)$ of $f(\Omega)$, there is a sequence $\{q_n\}_{n=1}^{\infty}$ in $\Omega$ such that
$\lim_{n\rightarrow\infty}f(q_n)=w$. Since $\overline{\Omega}$ is compact, we may assume that $\{q_n\}_{n=1}^{\infty}$ converges to a point, say $q_{\infty}\in \overline{\Omega}$. If $q_{\infty}\in\Omega$, then, by the continuity  of $f$, $w=f(q_{\infty})\in f(\Omega)$, which  contradicts with  the openness of $f(\Omega)$. Therefore, $q_{\infty}\in \partial \Omega$. This together with $(\ref{size-condition})$ implies  that
$$|w-f(a)|=\lim_{n\rightarrow +\infty} |f(q_n)-f(a)|\geq \liminf_{q\rightarrow\partial \Omega}|f(q)-f(a)|=s>0.$$
In other words, the boundary $\partial f(\Omega)$ of the open set $f(\Omega)$ lies outside of the ball $B\big(f(a), s\big)$. Consequently, $f(\Omega)$ must contain the ball $B(f(a), s)$.
\end{proof}
\begin{remark}
From the proof of  Proposition \ref{open} above, the result holds naturally under the conditions as described in Proposition \ref{open} for general setting $\mathbb{R}^{n}$, instead of quaternions  $\mathbb{H}\cong \mathbb{R}^{4}$.
\end{remark}
 The so-called \textit{Apollonius circle} also can be generalized trivially to the quaternionic setting.
\begin{lemma}[Apollonius]\label{Apollonius}
 Let $a,b\in \mathbb{H}$, and $\mathbb{R} \ni t\neq 1$. Then the set $\{q\in \mathbb{H}: |q-a|\leq t|q-b|\}$ is  the closed ball $\overline{B(c,r)}$  with center and radius given by
$$c=\frac{a-t^{2}b}{1-t^{2}}, \quad r=\frac{|a-b|}{1-t^{2}}t.$$
\end{lemma}

Let  $\mathcal{SR}(\mathbb{B},\mathbb{B})$ be  the class of regular function $f\in \mathcal{SR}(\mathbb{B})$  with values in $\mathbb{B}$. A typical example of $\mathcal{SR}(\mathbb{B},\mathbb{B})$ is a regular M\"{o}bius transformation of $\mathbb B$ onto $\mathbb B$ given by
$$f(q)=(1-q\overline{u})^{-\ast}\ast(q-u)v$$ with $u\in \mathbb B$ and $v\in \partial \mathbb B$ (cf. \cite[Corollary 9.17]{GSS}).

\begin{proposition}[Lindel\"{o}f]\label{L-inequality}
Let $f\in \mathcal{SR}(\mathbb{B},\mathbb{B})$. Then for all  $q\in\mathbb B$, the following inequalities hold:

\begin{equation}\label{20}
\Big|f(q)-\frac{1-|q|^2}{1-|q|^2|f(0)|^2}f(0)\Big|\leq\frac{|q|\big(1-|f(0)|^2\big)}{1-|q|^2|f(0)|^2};
\end{equation}

\begin{equation}\label{21}
\frac{|f(0)|-|q|}{1-|q||f(0)|}\leq|f(q)|\leq \frac{|q|+|f(0)|}{1+|q||f(0)|};
\end{equation}
\begin{equation}\label{22}
\big|f(q)-f(0)\big|\leq \frac{|q|\big(1-|f(0)|^2\big)}{1-|q||f(0)|}.
\end{equation}

Equality holds for one of inequalities in $(\ref{20})-(\ref{22})$ at some point
$q_0\in \mathbb B\setminus\{0\}$ if and only if $f$ is a regular M\"{o}bius transformation of $\mathbb B$ onto $\mathbb B$.
\end{proposition}

\begin{proof}
The Schwarz-Pick theorem (see \cite{BS}) gives, for all $q\in\mathbb B$,
\begin{equation}\label{30}
\big|\big(1-f(q)\overline{f(0)}\,\big)^{-\ast}\ast\big(f(q)-f(0)\big)\big|\leq |q|,
\end{equation}
which together with Proposition $\ref{prop:Quotient Relation}$ implies that
$$\frac{\big|f\circ T_{1-f\overline{f(0)}}(q)-f(0)\big|}{\big|1-f\circ T_{1-f\overline{f(0)}}(q)\overline{f(0)}\big|} \leq |q|.$$
According to Lemma \ref{Apollonius}, the preceding inequality is equivalent to
$$\Big|f\circ T_{1-f\overline{f(0)}}(q)-\frac{1-|q|^2}{1-|q|^2|f(0)|^2}f(0)\Big|
\leq\frac{|q|\big(1-|f(0)|^2\big)}{1-|q|^2|f(0)|^2}.$$
Since $f(\mathbb B)\subset \mathbb B$, it follows from  Proposition $\ref{prop:Quotient Relation}$ that $T_{1-f\overline{f(0)}}$ is a homeomorphism with inverse $T_{1-f(0)\ast f^c}$.  Replacing $q$ by $T_{1-f(0)\ast f^c}(q)$ in the preceding inequality gives that
$$\Big|f(q)-\frac{1-|q|^2}{1-|q|^2|f(0)|^2}f(0)\Big|
\leq\frac{|q|\big(1-|f(0)|^2\big)}{1-|q|^2|f(0)|^2},\qquad \forall\,q\in\mathbb B.$$

If equality achieves at some point $0\neq q_0\in \mathbb B$ in the last inequality, then it also achieves at point $0\neq\widetilde{q_0}=T_{1-f(0)\ast f^c}(q_0)$  in inequality $(\ref{30})$. It thus follows from the Schwarz-Pick theorem that $f$ is a regular M\"{o}bius transformations of $\mathbb B$ onto $\mathbb B$.

Inequalities in $(\ref{21})$ and $(\ref{22})$ follow from $(\ref{20})$ as well as  the triangle inequality.  The proof is complete.
\end{proof}

For each $p=x+yI\in \mathbb H\setminus \mathbb R$ with $x, y\in \mathbb R$ and $I\in\mathbb S$, we denote by $\mathbb S_p$ the 2-dimensional sphere given by
$$\mathbb S_p:=\big\{x+yJ: J\in\mathbb S\big\}.$$
The next result is a crucial step towards Theorem \ref{Main result-Landau}.
\begin{lemma}[Landau]\label{Landau}
Let $f\in \mathcal{SR}(\mathbb{B},\mathbb{B})$ with $f(0)=0$ and $f'(0)=\alpha>0$. Then
\begin{enumerate}
\item [(a)] $f$ is injective on $B(0, r_0)$, where
$$r_0=\frac{\alpha}{1+\sqrt{1-\alpha^2}};$$
\item  [(b)] for each positive number $r\leq r_0$, $f\big(0, B(r)\big)$ contains the ball $B\big(0, R(r)\big)$, where
$$R(r)=r\frac{\alpha-r}{1-\alpha r}\geq r r_0.$$
\end{enumerate}
\end{lemma}

\begin{proof}
Clearly, $\alpha\in (0,1]$. If $\alpha=1$, then $f(q)=q$ for all $q\in\mathbb B$ and the results hold trivially. Now we assume that  $\alpha\in (0,1).$  To prove ${\rm{(a)}}$, it suffices to show that if there exist two distinct points mapped by $f$ to one common point, then one of these two points must lie outside of the open ball $B(r_0).$ We proceed as follows. Suppose that there exists one point $q_0\in\mathbb B$ such that $\rho_0:=|q_0|>0$ and $f(q_0)=0$. From the maximum principle (see cf. \cite[Theorem 7.1]{GSS}) and $\alpha\in (0,1)$, it follows that the function $g(q):=q^{-1}f(q)$ is a regular self-mapping of $\mathbb B$. Applying the first inequality in (\ref{21}) to $g$ shows that the following inequality
\begin{equation}\label{SP-ineq}
|f(q)| \geq|q|\frac{\alpha-|q|}{1-\alpha|q|}
\end{equation}
holds for every $q\in\mathbb B$. Now it follows from   inequality   $(\ref{SP-ineq})$ with $q=q_0$ that
\begin{equation}\label{Radius01}
\rho_0\geq \alpha>r_0.
\end{equation}

Now suppose that there exist two distinct points $q_1, q_2\in\mathbb B$ such that $0<|q_1|\leq |q_2|=:\rho<1$ and $f(q_1)=f(q_2)=:w_0\neq 0$. We claim that
\begin{equation}\label{L-ineq}
|w_0|\leq \rho^2.
\end{equation}
Set
$$\psi(q)=\big(f(q)-w_0)\ast\big(1-f(q)\overline{w}_0\big)^{-\ast}.$$
We must consider the following two cases. The first case is that $q_1, q_2$ lie in a same 2-dimensional sphere, i.e. $\mathbb S_{q_1}=\mathbb S_{q_2}$.  In this case, the sphere $\mathbb S_{q_1}=\mathbb S_{q_2}$ is contained in the zero set of $\psi$, and hence the function
$$h(q):=\Big(\big((1-q\overline{q}_1)^{-\ast}\ast(q-q_1)\big)^s\Big)^{-1}\psi(q)$$
is regular on $\mathbb B$ and is bounded in modulus by one, in virtue of the maximum principle (see cf. \cite[Theorem 7.1]{GSS}). In particular, $|h(0)|=|w_0|/|q_1|^2\leq1$, and thus
$$|w_0|\leq |q_1|^2=\rho^2$$
as claimed.

The second case is that $q_1, q_2$ lie in two different 2-dimensional spheres, i.e. $\mathbb S_{q_1}\cap\mathbb S_{q_2}=\emptyset$. Set
$$\varphi(q)=(q-q_1)\ast(1-q\overline{q}_1)^{-\ast}, \qquad \textrm{and} \qquad \phi(q)=(q-\widetilde{q}_2)\ast(1-q\overline{\widetilde{q}}_2)^{-\ast},$$
where
$$\widetilde{q}_2=T_{\varphi^c}(q_2).$$
Then $\varphi\ast\phi$ has precisely two zeros $q_1$ and $q_2$.
Reasoning as before, $(\varphi\ast\phi)^{-\ast}\ast \psi$ is  also a regular function on $\mathbb B$ with values in $\overline{\mathbb B}$. In particular,
$$|w_0|/|q_1||\widetilde{q}_2|=\big|(\varphi\ast\phi)^{-\ast}\ast \psi(0)\big|\leq 1,$$
implying that
$$|w_0|\leq|q_1||\widetilde{q}_2|=|q_1||q_2|\leq \rho^2$$
as desired.

Now applying inequality $(\ref{SP-ineq})$ with $q=q_2$ once more, together with inequality $(\ref{L-ineq})$ implies that
$$\frac{\alpha-\rho}{1-\alpha\rho}\leq \rho,$$
which is equivalent to
\begin{equation}\label{Radius02}
\rho\geq r_0=\frac{\alpha}{1+\sqrt{1-\alpha^2}}.
\end{equation}
Now from inequalities (\ref{Radius01}) and (\ref{Radius02}) we conclude that $\left.f\right|_{B(0, r_0)}$ is injective. This completes the proof of ${\rm{(a)}}$.

From ${\rm{(a)}}$ and the open mapping theorem (see \cite[Theorem 7.7]{GSS}), it follows that
$$\left. f\right|_{B(0, r_0)}:B(0, r_0)\rightarrow f\big(B(0, r_0)\big)$$
 is a homeomorphism. Thus for an arbitrary fixed $r\in(0, r_0],$ we easily deduce from inequality  (\ref{SP-ineq}) that the boundary $\partial f\big( B(0, r)\big)$ of the open set $f\big( B(0, r)\big)$ (containing the origin $0$) lies outside of the ball $B\big(R(0, r)\big)$ and ${\rm{(b)}}$ immediately follows.
\end{proof}

\section{Proof of Theorems \ref{Bloch-Landau-I-type-Bloch}, \ref{Bloch-Landau-I-type-Bergman} and \ref{Main result-Landau}}

Now we use the lemmas established in the preceding section  to prove Theorems \ref{Bloch-Landau-I-type-Bloch}, \ref{Bloch-Landau-I-type-Bergman} and \ref{Main result-Landau}.

\begin{proof}[Proof of Theorem $\ref{Bloch-Landau-I-type-Bloch}$]
Recalling the classical Wolff-Warschawski-Noshiro theorem for holomorphic functions on convex domains in $\mathbb C$ (cf. \cite[lemma 2.4.1]{Graham-Kohr}), the assertion (a)  follows directly from the splitting lemma (Lemma   \ref{eq:Splitting})  and  Theorem \ref{Bonk}.

Now we turn to the assertion (b). We first prove  that
$$ B(0, \sqrt{3}/{4}) \subset f\big(B(0, 1/\sqrt{3})\big).$$ From Theorem \ref{open-map}, we deduce  that the image set $f\big(B(0, 1/\sqrt{3})\big)$ is open in $\mathbb H$. Since that $f(0)=0$, according to Proposition \ref{open}, it suffices to show that
$$\min_{|q|=\frac1{\sqrt{3}}}|f(q)|\geq \frac{\sqrt{3}}{4}.$$ Indeed, for each point $q=e^{I\theta}/\sqrt{3}$ with $I\in\mathbb S$ and $\theta\in \mathbb R$,
 from inequality (\ref{Bonk-ineq}) it follows  that
$$|f(q)|\geq {\rm{Re}}\,f(q) ={\rm{Re}}\,\int_{0}^{\frac{1}{\sqrt{3}}}f'(te^{I\theta})dt
\geq \int_{0}^{\frac{1}{\sqrt{3}}}\frac{1-\sqrt{3}t}{(1-\sqrt{1/3}t)^{3}}dt
=\frac{\sqrt{3}}{4}$$
as desired. Now it remains to prove the injectivity of $\left. f\right|_{B(0, r_{\mathbb B})}$ and the relation  $B(0, R_{\mathbb B})\subset  f\big(B(0, r_{\mathbb B})\big)$. To this end,  applying the fundamental theorem of calculus to $f$ yields that
\begin{equation}\label{growth-estimate01}
|f(q)|\leq\frac{1}{2}\log \frac{1+|q|}{1-|q|}
\end{equation}
for all $q\in \mathbb{B}$. For each $r\in (0, 1)$, consider the regular function $g$ given by $g(q)=f(rq)$, which satisfies both
$g(0)=0$ and $g'(0)=r$. Thus it follows from  Lemma \ref{Landau} that
$g$ is injective in  $B(0, \rho_{0})$ and $B(0, R)\subset g(\mathbb{B})$, where $$\rho_{0}=\frac{r}{\frac{1}{2}\log \frac{1+r}{1-r}+\sqrt{\frac{1}{4}\log ^{2}\frac{1+r}{1-r}-r^{2}}}$$ and $$R=\frac{1}{2}\log \frac{1+r}{1-r}\rho_{0}^{2}.$$
This means that
$f$ is injective in  $B(0, r\rho_{0})$ and $B(0, R)\subset f\big(B(0, r\rho_{0})\big)$ for all $0<r<1$.
Hence $f$ is injective in  $B(0, r_{\mathbb B})$ and $B(0, R_{\mathbb B})\subset f\big(B(0, r_{\mathbb B})\big)$, where
$$ r_{\mathbb B}=\sup_{0<r<1}\bigg\{ \frac{1}{2}\log \frac{1+r}{1-r}-\sqrt{\frac{1}{4}\Big(\log\frac{1+r}{1-r}\Big)^{2}-r^{2}} \bigg\}
\approx 0.3552,$$
and
$$R_{\mathbb B}=\sup_{0<r<1}\bigg\{\frac{1}{2r^{2}}\log \frac{1+r}{1-r}\bigg(\frac{1}{2}\log \frac{1+r}{1-r}-\sqrt{\frac{1}{4}\Big(\log\frac{1+r}{1-r}\Big)^{2}-r^{2}}\bigg)^{2}  \bigg\}\approx 0.2308.$$
\end{proof}

\begin{proof}[Proof of Theorem $\ref{Bloch-Landau-I-type-Bergman}$]
The proof is similar to that of Theorem $\ref{Bloch-Landau-I-type-Bloch}$. The only difference is that, instead of (\ref{growth-estimate01}), we use the following estimate:
\begin{equation}\label{growth-estimate02}
|f(q)|\leq\frac{1}{(1-|q|)^{\frac{2}{p}}}
\end{equation}
for every $f\in A^p(\mathbb B)$ and all $q\in \mathbb{B}$. The proof of inequality (\ref{growth-estimate02}) is standard, which goes as follows.
For $p\in [1,+\infty)$, by Lemma \ref{eq:Splitting}, $|f_{I}|^{p}$ is subharmonic on $\mathbb B_{I}$ for all $I\in \mathbb{S}$. Then, for $r\in [0,1-|z|),$ we have
 $$|f_{I}(z)|^{p}\leq \frac{1}{2\pi}\int_{0}^{2\pi}|f_{I}(z+re^{I\theta})|^{p}d\theta.$$
Integration gives that
 $$(1-|z|)^{2}|f_{I}(z)|^{p}
 \leq \frac{1}{\pi}\int_{0}^{1-|z|}\int_{0}^{2\pi}|f_{I}(z+re^{I\theta})
 |^{p}rd\theta dr,$$
 which together with  the prescribed condition $\|f\|_{A^{p}}\leq1$ implies that
 $$(1-|z|)^{2}|f_{I}(z)|^{p}
 \leq\int_{\mathbb{B}_{I}}|f_{I}|^{p}d\sigma_{I}\leq1$$
for all $ I\in \mathbb{S}$ and $z\in \mathbb{B}_{I}$. This leads to inequality (\ref{growth-estimate02}).
\end{proof}

\begin{proof}[Proof of Theorem $\ref{Main result-Landau}$]
With notations introduced in the introduction in mind, we consider the regular function $f_{\alpha}:\mathbb B\rightarrow \mathbb B$ given by
$$f_{\alpha}(q)=q\frac{q+\alpha}{1+q\alpha},$$
which belongs to $\mathcal{F}_{\alpha}$. Therefore, by the preceding theorem  the assertions ${\rm{(a)}}$ and ${\rm{(b)}}$ in Lemma \ref{Landau} holds for this function $f_{\alpha}$. Note also that the slice derivative function
$$f'_{\alpha}(q)=\frac{q^2\alpha+2q+\alpha}{(1+q\alpha)^2}$$
has a zero at $q=-r(\mathcal{F}_{\alpha})$ and \
$$f_{\alpha}(-r)=-r\frac{\alpha-r}{1-\alpha r}$$
for every $r\in (0, r(\mathcal{F}_{\alpha})]$. This fact together with Lemma \ref{Landau} implies the statements of ${\rm{(a)}}$ and ${\rm{(c)}}$ in this theorem.

To complete the proof, it remains to consider the extremal cases in ${\rm{(b)}}$ and ${\rm{(d)}}$. If $r_{\alpha}(f)=r(\mathcal{F}_{\alpha})$, then from the proof of Lemma \ref{Landau} it follows that equality holds for inequality in $(\ref{Radius02})$. Consequently, equality has to hold for inequality in $(\ref{SP-ineq})$ at the point $q=q_2$. By Proposition \ref{L-inequality}, this is possible only if $f$ is of the form in $(\ref{Extremal-Cover})$. Conversely, according to Lemma \ref{Landau}, each regular function  $f$ given in $(\ref{Extremal-Cover})$ is indeed injective on the ball $B\big(0, r(\mathcal{F}_{\alpha})\big)$. (It seems not easy to verify this fact directly.) Therefore, $r_{\alpha}(f)\geq r(\mathcal{F}_{\alpha})$.  Moreover, the slice derivative function  $f'$ has a zero at the point
$$q=r(\mathcal{F}_{\alpha})e^{I\theta}=\frac{\alpha}{1+\sqrt{1-\alpha^2}}e^{I\theta},$$
and hence the radius of the largest disc  $B(0, r)_I$ on which  $f_I$ is injective is precisely $r(\mathcal{F}_{\alpha})$. Thus we conclude that $r_{\alpha}(f)= r(\mathcal{F}_{\alpha})$ for each regular function  $f$ given in $(\ref{Extremal-Cover})$. This completes the proof of ${\rm{(b)}}$.  ${\rm{(d)}}$ can be proved similarly and its proof is left to the interested reader.
\end{proof}

\section{Proof of Theorem \ref{Bloch-Landau-II-type}}
In this section, we give a proof of  Theorem \ref{Bloch-Landau-II-type}. Given a constant $u\in \partial\mathbb{ B}$ and a regular function $f$ on $ \overline{\mathbb{B}} $ with power series expansion
$$f(q)=\sum_{n=0}^{\infty}q^na_n.$$
Then the \textit{slice regular rotation} $\widetilde{f}_{u}$ of $f$ is defined to be the regular function on $ \overline{\mathbb{B} }$ given by
$$ \widetilde{f}_{u}(q)=\sum\limits_{n=0}^{\infty}q^nu^{n}a_n.$$

Following the idea of Estermann in \cite{Estermann}, we can prove the following lemma.
\begin{lemma}\label{bloch-ball}
 Let $f\in \mathcal{SR}(\mathbb{B})$   be   nonconstant and satisfy
  \begin{equation}\label{der-condition}
  |f'(q)|\leq 2|f'(a)|
  \end{equation}
 for all $q\in \overline{B(a,r)}$ with some $a\in \mathbb B$ and $r\in (0, 1-|a|)$. Then there exists $u\in \partial\mathbb{B}$ such that
 $$B(f(a),R)\subset \widetilde{f}_{u}(\mathbb{B}),$$
 where $R=(5-2\sqrt{6})r|f'(a)|. $
\end{lemma}
\begin{proof}
Firstly, let us consider the case   that $a\in (-1,1)$.
Set
$$g(q)=f(q)-f(a)-(q-a)f'(a).$$  Then $g$ is a   regular function  on $\mathbb{B}$.
Applying the fundamental theorem of calculus to $g$ yields that, for all $q\in \mathbb{B}$,
\begin{equation}\label{der-estimate01}
\begin{split}
 g(q)
 =&\int_{0}^{1}\frac{d}{dt} g\big(tq+(1-t)a\big) dt
  \\
 =& (q-a)\int_{0}^{1}f'\big(tq+(1-t)a\big)-f'(a) dt.
 \end{split}
 \end{equation}
 Fix an arbitrary point $q=x+yI\in \overline{B(a,r)} \subset \mathbb{B}$ with  some $I\in \mathbb{S}$. By Lemma \ref{eq:Splitting} and Cauchy integral formula for holomorphic functions, the integrand in the second integral in (\ref{der-estimate01}) can be represented as
$$f'\big(tq+(1-t)a\big)-f'(a)=\frac{1}{2\pi I}\int_{\partial B(a,r)_{I}} \bigg(\frac{1}{z-\big(tq+(1-t)a\big)}-\frac{1}{z-a}\bigg)dzf'(z),$$
from which  and (\ref{der-condition}) it follows that, for all $t\in [0, 1]$,
$$ |f'\big(tq+(1-t)a\big)-f'(a)|\leq \frac{t|q-a|}{r-t|q-a|}
\max_{\overline{B(a,r)}}|f'|\leq \frac{2t|q-a|}{r-t|q-a|}|f'(a)|. $$
Therefore,
$$ |g(q)|\leq2|q-a|^{2}|f'(a)|\int_{0}^{1}\frac{t}{r-t|q-a|}dt
\leq\frac{|q-a|^{2}|f'(a)|}{r-|q-a|},$$
and hence
$$ |f(q)-f(a)|\geq |(q-a)f'(a)|-|g(q)|\geq |q-a|\Big(1-\frac{|q-a|}{r-|q-a|}\Big)|f'(a)|$$
for all $q\in \overline{B(a,r)}$. In particular,
the inequality
$$ |f(q)-f(a)| \geq (3-2\sqrt{2})r|f'(a)|$$
holds for all $q\in \partial B\big(a, (1-\sqrt{2}/2)r \big)$.
Note that $a\in (-1, 1)$,  $B(a,r)$ is an  axially  symmetric open  subset of  $\mathbb{B}$. Hence,  by Theorem \ref{open-map}, $f\big(B(a,r)\big)$ is open.
This together  with  Proposition \ref{open} implies that
\begin{eqnarray} \label{condition1}
B(f(a),R_{1})\subset f(B(a,r))\subset f(\mathbb{B}),
\end{eqnarray}
where $R_1=(3-2\sqrt{2})r|f'(a)|.$

Secondly, let us consider the case  that $a\in \mathbb{B}\setminus\mathbb{R}$.
Let  $u=\frac{a}{|a|}\in \partial\mathbb{B}_{I}$ with $I\in \mathbb{S}$ and consider  the   regular function $\widetilde{f}_{u}(q)$ on $\mathbb{B}$.
It is obvious that
$$\widetilde{f}_{u}(|a|)=f(a) \ \ \textrm{and} \ \ \widetilde{f}_{u}'(|a|)=uf'(a).$$

The representation formula for regular functions in \cite[Theorem 3.1]{CGSS} gives that
$$\widetilde{f}_{u}'(q)
=\frac{1}{2}(1-I_qI)\widetilde{f}'_{u}(z)+
\frac{1}{2}(1+I_qI)\widetilde{f}'_{u}(\overline{z}),$$
where $q=x+yI_q $ for some $I_q \in \mathbb{S} $ and $z=x+yI$.

Notice that  $  \widetilde{f}'_{u}(\overline{z})=uf'(\overline{z}u)$ and $\widetilde{f}'_{u}(z)=uf'(zu)$, then  we obtain that
$$ |\widetilde{f}_{u}'(q)|\leq |f'(zu)|+|f'(\overline{z}u)|, $$
and hence, for $|q-|a||\leq r$,  which implies  that  $|zu-a|\leq r$ and $|\overline{z}u-a|\leq r$,
\begin{eqnarray*}
 \max_{q\in \overline{B(|a|,r)}}|\widetilde{f}_{u}'(q)|
 &\leq&
\max_{zu\in \overline{B(a,r)}_{I}}(|f'(zu)| + \max_{\overline{z}u\in \overline{B(a,r)}_{I}} |f'(\overline{z}u)|)
 \\
 &\leq& 4|f'(a)|
  \\
 &=&4|\widetilde{f}_{u}'(|a|)|.
 \end{eqnarray*}
Here the second inequality follows from (\ref{der-condition}).

Now a similar argument as in the first case gives that
$$ |\widetilde{f}_{u}(q)-\widetilde{f}_{u}(|a|)|\geq |q-|a||\Big(1-\frac{2|q-|a||}{r-|q-|a||}\Big)|\widetilde{f}_{u}'(|a|)|, \qquad \forall\, q\in \overline{B(|a|, r)}.$$
In particular, for any $q\in \partial B\big(|a|, (1-\sqrt{6}/3)r \big)$,
$$ |\widetilde{f}_{u}(q)-\widetilde{f}_{u}(|a|)|\geq (5-2\sqrt{6})r|\widetilde{f}_{u}'(|a|)|
= (5-2\sqrt{6})r|f'(a)|=:R_{2}.$$
 By   Proposition \ref{open} and  Theorem \ref{open-map} again, we have
$$B(\widetilde{f}_{u}(|a|),R_{2})\subset \widetilde{f}_{u}(B(|a|,r))\subset \widetilde{f}_{u}(\mathbb{B}), $$
i.e.,
\begin{eqnarray} \label{condition2}
 B(f(a),R_{2})\subset \widetilde{f}_{u}(\mathbb{B}).
 \end{eqnarray}

Now the desired result follows immediately from (\ref{condition1}) and (\ref{condition2}).
\end{proof}

Finally we have all the tools to prove the second  Bloch  type theorem
for regular functions.
\begin{proof}[Proof of Theorem \ref{Bloch-Landau-II-type}]
Let $f$ be as described. We consider    the continuous function $\psi$ on $\overline{\mathbb{B}}$ given by
$$\psi(q)=(1-|q|)|f'(q)|, $$
which vanishes on the boundary $\partial\mathbb{B}$. Then there exists some
$\ a\in \mathbb{B}$ such that
$$\psi(a)=\max_{q\in \overline{\mathbb{B}}}\psi(q).$$
Set $r:=\frac{1}{2}(1-|a|)$. Then it is evident that
$$\overline{B(a,r)}\subset \mathbb{B}, \qquad  r|f'(a)|\geq \frac{1}{2}\psi(0)=\frac{1}{2},$$
 and
  $$r\leq 1-|q|, \qquad \forall\,  q\in\overline{ B(a,r)}.$$
Whence
$$|f'(q)|\leq 2|f'(a)|, \qquad \forall\,  q \in \overline{B(a,r)}\subset \mathbb{B}.$$
From Lemma \ref{bloch-ball}, there exists $u\in \partial\mathbb{B}$ such that
 $$B(f(a),R)\subset \widetilde{f}_{u}(\mathbb{B}), $$
 where $R=(5-2\sqrt{6})r|f'(a)|\geq \frac{5}{2}-\sqrt{6}.$
The proof is complete.
\end{proof}

\bibliographystyle{amsplain}

\end{document}